\documentclass[12pt, notitlepage]{amsart}
\usepackage{latexsym, amsfonts, amsmath, amssymb, amsthm}

\newtheorem{theorem}{Theorem}[section]
\newtheorem{lemma}[theorem]{Lemma}
\newtheorem{proposition}[theorem]{Proposition}

\newtheorem*{nonumtheorem}{Theorem}

\newtheorem{remark}[theorem]{Remark}
\newtheorem{definition}[theorem]{Definition}
\newcommand{\Z}{\mathbb{Z}}
\newcommand{\R}{\mathbb{R}}
\newcommand{\E}{\mathbb{E}}
\newcommand{\T}{\mathbb{T}}

\usepackage{graphicx}
\usepackage{mathrsfs}

\begin{document}

\title{A New Proof of Vinogradov's Three Primes Theorem}
\author{Xuancheng Shao}

\address{Department of Mathematics \\ Stanford University \\
450 Serra Mall, Bldg. 380\\ Stanford, CA 94305-2125}
\email{xshao@math.stanford.edu}

\begin{abstract}
We give a new proof of Vinogradov's three primes theorem, which
asserts that all sufficiently large odd positive integers can be
written as the sum of three primes. Existing proofs rely on the
theory of $L$-functions, either explicitly or implicitly. Our proof
uses instead a transference principle, the idea of which was first
developed by Green \cite{green_roth}. To make our argument work, we
also develop an additive combinatorial result concerning popular
sums, which may be of independent interest.
\end{abstract}

\maketitle

\section{Introduction}

In this paper, we study additive problems involving primes. The
famous Goldbach conjecture asserts that every even positive integer
at least $4$ is the sum of two primes. Although the binary Goldbach
problem is considered to be beyond the scope of current techniques,
its ternary analogue has been settled by Vinogradov
\cite{vinogradov} in 1937.

\begin{nonumtheorem}[Vinogradov]
Every sufficiently large odd positive integer can be written as the sum of three primes.
\end{nonumtheorem}

The classical approach to Vinogradov's theorem is to use the circle
method. For positive integers $N$, we write $r(N)$, the number of
representations of $N$ as the sum of three primes, as an integral
\[ r(N)=\int_0^1f(\alpha)e(-\alpha N)d\alpha, \]
where $f(\alpha)$ is an exponential sum over primes:
\[ f(\alpha)=\sum_{p\leq N}e(\alpha p). \]
To estimate this integral, dissect $[0,1]$ into major arcs
$\mathfrak{M}$ and minor arcs $\mathfrak{m}$. Roughly speaking, the
major arcs $\mathfrak{M}$ consist of those $\alpha$ that are very
close to a fraction $a/q$ with a small denominator $q$. For these
$\alpha$, we can obtain asymptotics for $f(\alpha)$ using the prime
number theorem in arithmetic progressions modulo $q$. For
$\alpha\in\mathfrak{m}$ in the minor arc, we expect enough
cancellation in the exponential sum so that $f(\alpha)$ is small in
magnitude. Combining the major arc and the minor arc analysis, one
can deduce an asymptotic formula for $r(N)$, with the main term
coming from the major arcs.

In the major arc analysis, we appealed to the prime number theorem
in arithmetic progressions. These arithmetic progressions have
length approximately $N$ and steps up to some power of $\log N$. In
this regime, the prime number theorem in arithmetic progressions is
the Siegel-Walfisz's theorem, whose proof relies on the theory of
$L$-functions (see Chapter 22 of \cite{davenport}). Certain implied
constants in the Siegel-Walfisz's theorem are not effective due to
the possibility of the existence of Siegel zeros. Heath-Brown
\cite{heath-brown} gave a different proof of Vinogradov's theorem by
directly using some identities involving primes, but his method also
requires Siegel-Walfisz's theorem.

Vinogradov's method can be made effective; that is, an explicit
constant $V$ can be obtained from the proof such that any odd
integer $N\geq V$ is the sum of three primes. The current record is
by Liu and Wang \cite{liu-wang}, claiming that $V$ can be taken to
be approximately $\exp(3100)$; see \cite{chen-wang} for the previous
best bound. See also \cite{helfgott} for a recent result. The bound
$V\leq\exp(3100)$ was obtained by incorporating various explicit
estimates for exponential sums and for possible Siegel zeros.

In this paper, we present an alternative proof of Vinogradov's
theorem, which avoids the theory of $L$-functions. In particular, an
explicit bound for $V$ can be extracted from our method by keeping
track of the implied constants. Unfortunately, directly doing so
would produce a bound for $V$ of triple exponential type:
$V\leq\exp(\exp(\exp(C)))$, where $C$ is a reasonable constant. This
is far from the best record (but see Remark \ref{rem:bound}).
Nevertheless, we believe that our new proof is still interesting in
its own right.

Our method uses the idea of the transference principle, first
developed by Green \cite{green_roth} in his proof of Roth's theorem
in the primes. This idea has become a powerful tool for studying
additive problems in dense subsets of primes. In the setting of the
ternary Goldbach problem, we state the following transference
principle obtained by following Green's argument directly (see
\cite{li} and \cite{shao}). For the precise definitions in the
pseudorandomness condition and the $L^q$ extension estimate, see
Definition \ref{def:trans} below.

\begin{theorem}\label{thm:1/2}
Let $0<\delta<1$ be given. Then for sufficiently small $\eta>0$ and
sufficiently large prime $N$, the following statement holds. For
$i=1,2,3$, let $\nu_i,a_i:\Z/N\Z\rightarrow\R$ be arbitrary
functions. Let $\alpha_i$ be the average of $a_i$. Suppose that they
satisfy the following assumptions:
\begin{enumerate}
\item (majorization condition) $0\leq a_i(n)\leq\nu_i(n)$ for all $n\in\Z/N\Z$.
\item (mean condition) $\alpha_i\geq\delta$ and $\alpha_1+\alpha_2+\alpha_3\geq 1+\delta$.
\item (pseudorandomness condition) The majorant $\nu_i$ is
$\eta$-pseudorandom.
\item ($L^q$ extension estimate) The function $a_i$ satisfies the $L^q$ extension estimate for some $2<q<3$.
\end{enumerate}
Then for any $n\in\Z/N\Z$,
\[ \sum_{\substack{n_1,n_2,n_3\\ n_1+n_2+n_3\equiv n\pmod N}}a_1(n_1)a_2(n_2)a_3(n_3)\geq cN^2, \]
where $c=c(\delta)>0$ is a constant depending only on $\delta$.
\end{theorem}

Note that in the case when $\alpha_1=\alpha_2=\alpha_3=\alpha$, the
threshold for the average $\alpha$ is $1/3$.

\begin{remark}\label{rem:1/2}
The conclusion above counts the number of solutions to
$n_1+n_2+n_3\equiv n\pmod N$, while we are interested in solutions
to $n_1+n_2+n_3=n$ in the integers. For $n$ close to $N$, this
discrepancy can be resolved by demanding the function $a_i$ to be
supported in the interval $[0,2N/3]$. In doing so, however, we are
effectively reducing the average of $a_i$ by a factor of $2/3$, and
thus in applications the threshold for the average of $a_i$ is $1/2$
rather than $1/3$.
\end{remark}

The proof of Theorem \ref{thm:1/2} consists of two parts. In the
first part, Fourier analysis techniques are employed to convert the
problem from an arbitrary pseudorandom majorant $\nu_i$ to the case
$\nu_i=1$. This step is where the pseudorandomness condition and the
$L^q$ extension estimate are used. In the second part, the case
$\nu_i=1$ is treated, which follows from a quantitative
Cauchy-Davenport-Chowla inequality.

The transference principle is usually applied in the study of
additive problems involving dense subsets of primes. In such
applications, one can think of $\nu_i$ as the (normalized)
characteristic function of the primes, and $a_i$ as the (normalized)
characteristic function of a dense subset of the primes. The $L^q$
extension estimate for $a_i$ can be established in various ways (see
Remark \ref{rem:Lq} below). The pseudorandomness condition for
$\nu_i$ depends on the equidistribution of primes in arithmetic
progressions; in particular, Siegel-Walfisz theorem.

%In usual applications of such transference principle, we take $\nu$
%to be the (normalized) characteristic function of the primes and $a$
%to be the (normalized) characteristic function of a dense subset of
%the primes. This leads to a density version of Vinogradov's theorem
%(see \cite{shao}).

%Let $\mathcal{P}$ be the set of all primes. For a subset
%$A\subset\mathcal{P}$, the lower density of $A$ in $\mathcal{P}$ is
%defined by
%\[ \underline{\delta}(A)=\liminf_{N\rightarrow\infty}\frac{|A\cap [1,N]|}{|\mathcal{P}\cap [1,N]|}. \]
%Furthermore, for $i\in\{1,2\}$, the lower density of $A$ in the
%primes congruent to $i$ modulo $3$ is defined by
%\[ \underline{\delta}_i(A)=\liminf_{N\rightarrow\infty}\frac{|\{p\in A\cap [1,N]:p\equiv i\pmod 3\}}{|\{p\in\mathcal{P}\cap [1,N]:p\equiv i\pmod 3\}}. \]

%\begin{corollary}\label{cor:1/2}
%Let $A\subset\mathcal{P}$ be a subset of the primes such that
%$2\underline{\delta}_1(A)+\underline{\delta}_2(A)>3/2$ and
%$2\underline{\delta}_2(A)+\underline{\delta}_1(A)>3/2$. Assume
%further that for any positive integer $W$, and any reduced residue
%$b\pmod W$, $A$ contains a prime in the residue class $b\pmod W$.
%Then for sufficiently large odd positive integers $N$, there exist
%$a_1,a_2,a_3\in A$ with $N=a_1+a_2+a_3$.
%\end{corollary}

%In the case when $\underline{\delta}_1(A)=\underline{\delta}_2(A)$,
%the assumption on the density of $A$ becomes
%$\underline{\delta}(A)>1/2$. Certainly, this corollary implies
%Vinogradov's theorem.

In this paper, we will choose the majorant $\nu_i$ differently so
that its pseudorandomness can be established elementarily. To prove
Vinogradov's theorem, we attempt to take $a_i$ to be the
(normalized) characteristic function of the primes and $\nu_i$ to be
Selberg's majorant for the primes. Since Selberg's majorant can be
expressed as a (relatively short) sum of standard multiplicative
functions, its pseudorandomness can be proved using elementary
estimates involving these multiplicative functions. Our plan of
deducing Vinogradov's theorem without using the theory of
$L$-functions is to use the transference principle with this new
choice of $a_i$ and $\nu_i$.

However, Theorem \ref{thm:1/2} does not quite serve our purpose. The
parity phenomenon in sieve theory suggests that the mean value of
the majorant $\nu_i$ is necessarily more than twice the mean value
of the characteristic function of the primes. Thus Theorem
\ref{thm:1/2} barely fails to apply to this choice of $a_i$ and
$\nu_i$ (see Remark \ref{rem:1/2}). For an excellent account of
sieve theory including the parity phenomenon, see the book
\cite{opera}.

The main innovation of the current paper is a new version of the
transference principle, which applies even when the average of $a_i$
is slightly less than $1/2$. For the precise definitions in the
pseudorandomness condition, $L^q$ extension estimate, and the
regularity condition, see Definition \ref{def:trans} below.

\begin{theorem}\label{thm:2/5}
Let $0<\delta,\kappa<1$ be given. Then for sufficiently small
$\eta>0$ and sufficiently large positive integer $N$, the following
statement holds. Let $N_3=N$ and $N_1=N_2=\lfloor N/2\rfloor$. For
$i=1,2,3$, let $\nu_i,a_i:[1,N_i]\rightarrow\R$ be arbitrary
functions. Let $\alpha_i$ be the average of $a_i$. Suppose that they
satisfy the following assumptions:
\begin{enumerate}
\item (majorization condition) $0\leq a_i(n)\leq\nu_i(n)$ for all $1\leq n\leq N_i$.
\item (mean condition) $\alpha_i\geq\delta$ and $\tfrac{1}{2}(\min(1,\alpha_1+\alpha_2)+\alpha_2)+\alpha_3\geq 1+\delta$.
\item (pseudorandomness condition) The majorant $\nu_i$ is
$\eta$-pseudorandom.
\item ($L^q$ extension estimate) The function $a_i$ satisfies the $L^q$ extension estimate for some $2<q<3$.
\item (regularity condition for $a_1$) The function $a_1$ is
$(\delta/50,\kappa)$-regular.
\end{enumerate}
Then
\[ \sum_{\substack{n_1,n_2,n_3\\ n_1+n_2+n_3=N}}a_1(n_1)a_2(n_2)a_3(n_3)\geq cN^2, \]
where $c=c(\delta,\kappa)>0$ is a constant depending only on
$\delta$ and $\kappa$.
\end{theorem}

In the case when $\alpha_1=\alpha_2=\alpha_3=\alpha$, the threshold
for the average $\alpha$ is $2/5$. This threshold is now below $1/2$
and thus Theorem \ref{thm:2/5} can be applied to our new choice of
$a_i$ and $\nu_i$ described above.

Compared with Theorem \ref{thm:1/2}, we now work directly in $\Z$
rather than in $\Z/N\Z$. This requires some modifications to the
traditional argument. First, we no longer have certain tools such as
Bohr sets in the setting of finite abelian groups, which play a
crucial part in the traditional analysis. Second, Theorem
\ref{thm:2/5} becomes nontrivial even in the case $\nu_i=1$, as we
shall discuss in Section \ref{sec:pop}, and the additional
regularity condition is necessary for its truth.

\begin{remark}\label{rem:bound}
The dependence of $\eta$ on $\delta$ in Theorem \ref{thm:1/2} and on
$\delta,\kappa$ in Theorem \ref{thm:2/5} is exponential. In the
application to Roth's theorem in the primes, this causes an extra
layer of logarithm in the lower bound for the density threshold.
However, this extra layer of logarithm was removed by Helfgott and
de Roton \cite{helfgott-deroton}. It is possible that the same thing
can be done here as well.
\end{remark}

The rest of the article is organized as follows. In Section
\ref{sec:pop}, we treat an additive combinatorial problem concerning
popular sums, which could be of independent interest. This problem
stems from the $\nu_i=1$ case of Theorem \ref{thm:2/5}. In Section
\ref{sec:trans}, we combine this additive combinatorial result with
a modification of traditional arguments to prove Theorem
\ref{thm:2/5}. In Section \ref{sec:selberg}, we review the
construction of Selberg's majorant. In Section \ref{sec:random}, we
show that Selberg's majorant is pseudorandom. Finally in Section
\ref{sec:proof}, we deduce Vinogradov's theorem from Theorem
\ref{thm:2/5}.

\vspace{5 mm} \textbf{Acknowledgement.} The author would like to
express his gratitude to Ben Green for sharing with him the idea of
proving Vinogradov's theorem using a transference principle, and
also to his advisor, Kannan Soundararajan, for carefully reading
early drafts of the paper and providing many useful comments.

%%%%%%%%%%%%%%%%%%%%%%%%%%%%%%%%%%%%%%%%%%%%%%%%%%%%%%%%%%%%
%%%%%%%%%%%%%%%%%%%%%  POPULAR SUMS %%%%%%%%%%%%%%%%%%%%%%%%
%%%%%%%%%%%%%%%%%%%%%%%%%%%%%%%%%%%%%%%%%%%%%%%%%%%%%%%%%%%%

\section{Generalization of Freiman's $3k-3$ theorem to popular sums}\label{sec:pop}

In this section, we prove a combinatorial result related to the
$\nu_i=1$ case of Theorem \ref{thm:2/5}. Consider the case when
$\nu_i$ is the constant function $1$ and $a_i$ is the characteristic
function of some subset $A_i\subset [1,N_i]$ (Recall that
$N_1=N_2=\lfloor N/2\rfloor$ and $N_3=N$). Theorem \ref{thm:2/5}
claims that if the density of $A_i$ is larger than $2/5$, and if
$A_1$ satisfies some regularity condition, then $N$ can be written,
in many ways, as $a_1+a_2+a_3$ with $a_i\in A_i$. This is certainly
false without the regularity condition: for example, take $A_i$ to
be the set of consecutive integers starting from $1$.

As an important step towards this conclusion, we need to study the
problem of obtaining lower bounds on the number of popular sums in
the sumset $A_1+A_2$. More precisely, for $s\in A_1+A_2$, let $r(s)$
be the number of ways to write $s=a_1+a_2$ with $a_1\in A_1$ and
$a_2\in A_2$. We are interested in lower bounds on the cardinality
of the set
\[ D_K(A_1,A_2)=\{s\in A_1+A_2:r(s)\geq K\}. \]
Note that for $K=1$, $D_1(A_1,A_2)$ is simply the sumset $A_1+A_2$.
However, we are interested in the regime where $K$ is a small
positive constant times the cardinality of $A_1$ or $A_2$.

In this direction, Green and Ruzsa \cite{green_ruzsa} obtained the
following generalization of Kneser's theorem in arbitrary finite
abelian groups.

\begin{lemma}[Green and Ruzsa]\label{lem:gr}
Let $G$ be a finite abelian group. Let $D=D(G)$ be the size of the
largest proper subgroup of $G$. Let $A_1,A_2\subset G$ be subsets
and $K>0$ be a parameter. Suppose that
$\min(|A_1|,|A_2|)\geq\sqrt{K|G|}$. Then
\[ |D_K(A_1,A_2)|\geq \min(|G|,|A_1|+|A_2|-D)-3\sqrt{K|G|}. \]
\end{lemma}

When $G$ is a cyclic group, this is almost sharp when $A_1$ and
$A_2$ are arithmetic progressions of the same step. For our
purposes, we would like better bounds once these extreme cases are
excluded. For $A_1,A_2\subset\Z$, Freiman \cite{freiman} has shown
that the lower bound for $|A_1+A_2|=D_1(A_1,A_2)$ can be improved if
the diameters of $A_1$ and $A_2$ are large compared to $|A_1|$ and
$|A_2|$. For $A\subset\Z$, we define the diameter of $A$ to be the
smallest $d$ such that $A$ is contained in an arithmetic progression
of length $d$.

\begin{theorem}[Freiman]\label{thm:3k-3}
Let $A_1,A_2\subset\Z$ be finite sets with diameters $d_1,d_2$, respectively. Suppose that $d_1\leq d_2$. Then
\[ |A_1+A_2|\geq\min(|A_1|+d_2,2|A_1|+|A_2|-3). \]
\end{theorem}

When $A_1=A_2=A$ and $|A|=k$, the lower bound above reads $|A+A|\geq
3k-3$ if the diameter of $A$ is large. For this reason, it is
traditionally called Freiman's $3k-3$ theorem.

Our main result in this section is a generalization of Theorem
\ref{thm:3k-3} to popular sums, which essentially states that the
same lower bound above holds for $D_K(A_1,A_2)$ when $K=\gamma N$
for some small $\gamma>0$, under some regularity assumption on
$A_1$. Before stating the result, we first describe this regularity
condition. For $y\geq 2$, let $P(y)$ be the product of all primes up
to $y$.

\begin{definition}\label{def:reg}
Let $0<\beta,\kappa<1$ be parameters. A subset $A\subset [1,N]$ is
said to be $(\beta,\kappa)$-regular if \[ |\{(u,v)\in A\times
A:u\leq\beta N,v\geq (1-\beta)N,(v-u,P(\beta^{-1}))=1\}|\geq \kappa
N^2.
\]
\end{definition}

Roughly speaking, this regularity condition on $A$ ensures that the
diameter of $A$ is approximately $N$, even if a small proportion of
elements are removed from $A$. This definition is compatible with
the $(\beta,\kappa)$-regularity of the characteristic function of
$A$ (see Definition \ref{def:trans} below). We now state our main
result in this section.

\begin{theorem}\label{thm:pop}
Let $\beta,\kappa>0$ be parameters with $\beta<1/6$. Let
$A_1,A_2\subset [1,N]$ be arbitrary subsets with $|A_i|\geq 4\beta
N$ ($i=1,2$). Suppose that $A_1$ is $(\beta,\kappa)$-regular. Then
for $\gamma<\min(\kappa^2/(16\beta^2),\beta^2/16)$,
\[ |D_{\gamma N}(A_1,A_2)|\geq \min(N,|A_1|+|A_2|)+|A_2|-9\beta N. \]
\end{theorem}

Our argument is motivated by Lev and Smelianski's proof \cite{lev}
of Theorem \ref{thm:3k-3}. We embed the sets $A_1$ and $A_2$ in an
appropriately chosen cyclic group and then use Lemma \ref{lem:gr}.

%Before proving it, we first give a sketch of the proof of Theorem
%\ref{thm:3k-3} following Lev and Smelianski \cite{lev} (see also
%Section 3.4 of \cite{ruzsa}), as our proof of Theorem \ref{thm:pop}
%is motivated by their argument. By translating and rescaling, we may
%assume that the smallest element and the largest element of $A_i$
%are $0$ and $d_i$, respectively. Let $A_i'$ be the image of $A_i$ in
%$\Z/d_2\Z$. Then $|A_2'|=|A_2|-1$ and $|A_1'|\geq |A_1|-1$. Observe
%that
%\begin{equation}\label{1}
%|A_1+A_2|\geq |A_1'+A_2'|+|A_1|.
%\end{equation}
%The additional term $|A_1|$ above comes from the fact that for any $a_1\in A_1$, both $a_1$ and $a_1+d_2$ are in $A_1+A_2$, but they are only counted once in $|A_1'+A_2'|$. At this stage, we make an assumption that $d_2$ is prime to simplify matters. Then by Cauchy-Davenport,
%\[ |A_1+A_2|\geq\min(d_2,|A_1'|+|A_2'|-1)+|A_1|\geq\min(|A_1|+d_2,2|A_1|+|A_2|-3), \]
%giving the desired lower bound. If $d_2$ is composite, one needs to
%upgrade \eqref{1} using some other combinatorial observations, and
%use Kneser's theorem for $\Z/d_2\Z$. This part of the argument is a
%bit sophisticated, although completely elementary. We omit the
%details since they are not used in our proof of Proposition
%ref{prop:3k-3}.

%We now prove Proposition \ref{prop:3k-3}, using Lev and Smelianski's
%idea of embedding $A_1,A_2$ into finite abelian groups and appealing
%to results in the finite setting (namely Lemma \ref{lem:gr}).

\begin{proof}
Consider the bipartite graph $\Gamma=(A_1,A_2,E)$, whose vertices
are elements of $A_1$ and $A_2$, and whose edges are those pairs
$(a_1,a_2)$ ($a_1\in A_1,a_2\in A_2$) with $a_1+a_2\in D_{\gamma
N}(A_1,A_2)$. Since every element $s\in (A_1+A_2)\setminus D_{\gamma
N}(A_1,A_2)$ yields at most $\gamma N$ edges in the complement of
$\Gamma$, the edge set $E$ contains all but at most $\gamma N\cdot
|A_1+A_2|\leq 2\gamma N^2$ pairs.

Let $A_1'\subset A_1$ be the set of vertices in $A_1$ with degree at least $|A_2|-\sqrt{\gamma}N$. Then
\[ |A_1\setminus A_1'|\leq\frac{2\gamma N^2}{\sqrt{\gamma}N}\leq 2\sqrt{\gamma}N. \]

By hypothesis, there are at least $\kappa N^2$ pairs $(u,v)\in
A_1\times A_1$ with $u\leq\beta N$ and $v\geq (1-\beta)N$ such that
$(v-u,P(\beta^{-1}))=1$. The number of those pairs with either
$u\notin A_1'$ or $v\notin A_1'$ is bounded above by
$4\beta\sqrt{\gamma}N^2$, which is less than $\kappa N^2$ by the
choice of $\gamma$. Hence there exists such a pair with $u,v\in
A_1'$. Let $A_1''=A_1'\cap [u,v]$. Then
\[ |A_1''|\geq |A_1'|-2\beta N\geq |A_1|-2(\beta+\sqrt{\gamma})N. \]

Let $d=v-u$ be the difference between the largest and the smallest
elements of $A_1''$. Let $B_1,B_2$ be the images of $A_1'',A_2$,
respectively, under the projection map $\Z\rightarrow\Z/d\Z$. Then
$|B_1|=|A_1''|-1$ and $|B_2|\geq |A_2|-2\beta N$. We claim that
\begin{equation}\label{embed}
|D_{\gamma N}(A_1'',A_2)|\geq |D_{3\gamma
N}(B_1,B_2)|+(|A_2|-2\sqrt{\gamma}N). \end{equation} In fact, for
each popular sum $\bar{s}\in D_{3\gamma N}(B_1,B_2)\subset\Z/d\Z$,
there are at most three different ways to lift $\bar{s}$ to an
integer $s\in A_1''+A_2$ (since $\beta<1/6$ and thus $d>2N/3$). At
least one of those liftings lies in $D_{\gamma N}(A_1'',A_2)$. The
additional term $|A_2|-2\sqrt{\gamma}N$ accounts for the fact that,
for all but at most $2\sqrt{\gamma}N$ values of $a_2\in A_2$, both
sums $u+a_2$ and $v+a_2$ lie in $D_{\gamma N}(A_1'',A_2)$, but they
are the same modulo $d$.

It is easy to check that $|B_i|\geq\sqrt{3\gamma}N$. We may thus
apply Lemma \ref{lem:gr} to the sets $B_1,B_2$ inside $G=\Z/d\Z$ to
conclude that
\[ |D_{3\gamma N}(B_1,B_2)|\geq\min(d,|B_1|+|B_2|-D)-6\sqrt{\gamma}N, \]
where $D=D(\Z/d\Z)$ is the size of the largest subgroup of $\Z/d\Z$.
It follows from $(d,P(\beta^{-1}))=1$ that $D\leq\beta N$. Combining
this with the lower bounds for $|B_1|$, $|B_2|$, and $d$, we get
\[ |D_{3\gamma N}(B_1,B_2)|\geq\min(N,|A_1|+|A_2|)-(6\beta+8\sqrt{\gamma})N. \]
Hence by \eqref{embed},
\[ |D_{\gamma N}(A_1,A_2)|\geq |D_{\gamma N}(A_1'',A_2)|\geq\min(N,|A_1|+|A_2|)+|A_2|-(6\beta+10\sqrt{\gamma})\beta N. \]
This is enough to conclude the proof by the choice of $\gamma$.
\end{proof}

\begin{remark}
A central topic in additive combinatorics is the study of structures
of sets with small doubling. For $A\subset\Z$, the doubling of $A$
is the quantity $K=|A+A|/|A|$. Freiman's celebrated theorem gives a
classification of the sets with small doubling $K$: they are dense
subsets of generalized arithmetic progressions of rank at most $K$.
See \cite{tao-vu} for the precise result and its history. Theorems
\ref{thm:3k-3} and \ref{thm:pop} roughly states that, if $K<3$, then
$A$ is efficiently covered by an arithmetic progression. This gives
a more precise structure than Freiman's theorem when $K<3$. In the
wider region $K<4$, see \cite{eberhard-green-manners} for a recent
result.
\end{remark}

%%%%%%%%%%%%%%%%%%%%%%%%%%%%%%%%%%%%%%%%%%%%%%%%%%%%%%%%%%%%%
%%%%%%%%%%%%%%%%%% MORE on popular sums %%%%%%%%%%%%%%%%%%%%%
%%%%%%%%%%%%%%%%%%%%%%%%%%%%%%%%%%%%%%%%%%%%%%%%%%%%%%%%%%%%%

%\section{More on popular sums}

%For a finite subset $A\subset\Z$ and a positive integer $L<|A|$, let
%$d(A,L)$ be the smallest diameter among all subsets $B\subset A$
%with $|A\setminus B|=L$. One can hope for something along this line
%(for simplicity, I state the result with just one set $A$):

%\begin{theorem}
%Let $A\subset\Z$ be a finite set with $|A|=M$. Let $\alpha\in (0,1)$
%be real. Then there exists a positive constant
%$\gamma=\gamma(\alpha)$ so that
%\[ |D_{\gamma M}(A,A)|\geq\min(|A|+d(A,\alpha M),3|A|)-O(\alpha M). \]
%\end{theorem}

%Let $G\subset A\times A$ be the set of pairs $(a_1,a_2)$ such that
%$a_1+a_2\in D_{\gamma}(A,A)$. Without loss of generality, assume
%that the degree of any $a\in A$ in $G$ is at least
%$(1-\sqrt{\gamma})|A|$. Let $d$ be the diameter of $A$. Assume that
%\[ A=\{0=a_1<a_2<\cdots<a_M=d\} \]
%with $\text{gcd}(a_1,a_2,\cdots,a_M)=1$. Let $A'$ be the image of
%$A$ in $\Z/d\Z$. Then
%\[ |A+_GA|\geq |A'+_GA'|+(1-2\sqrt{\gamma})|A|\geq D_{\gamma M}(A')+(1-2\sqrt{\gamma})M. \]

%%%%%%%%%%%%%%%%%%%%%%%%%%%%%%%%%%%%%%%%%%%%%%%%%%%%%%%%%%%%%%%%%%%
%%%%%%%%%%%%%%%%%%%%%  TRANSFERENCE PRINCIPLE %%%%%%%%%%%%%%%%%%%%%
%%%%%%%%%%%%%%%%%%%%%%%%%%%%%%%%%%%%%%%%%%%%%%%%%%%%%%%%%%%%%%%%%%%

\section{The transference principle}\label{sec:trans}

In this section, we prove Theorem \ref{thm:2/5}. The precise
definitions of the pseudorandomness condition, $L^q$ extension
estimate, and the regularity condition are given as follows. For a
(compactly supported) function $f:\Z\rightarrow\R$, its Fourier
transform is defined by
\[ \hat{f}(\theta)=\sum_{n\in\Z}f(n)e(n\theta), \]
where $e(n\theta)=\exp(2\pi in\theta)$. The $L^q$ norm of its
Fourier transform is defined by
\[
\|\hat{f}\|_q=\left(\int_0^1|\hat{f}(\theta)|^qd\theta\right)^{1/q}.
\]
For $y\geq 2$, let $P(y)$ be the product of all primes up to $y$.

\begin{definition}\label{def:trans}
Let $f:[1,N]\rightarrow\R$ be an arbitrary function.
\begin{enumerate}
\item The function $f$ is said to be $\eta$-\textit{pseudorandom} if
$|\hat{f}(r/N)-\delta_{r,0}N|\leq\eta N$ for each $r\in\Z/N\Z$,
where $\delta_{r,0}$ is the Kronecker delta.
\item The function $f$ is said to satisfy the $L^q$ extension
estimate if $\|\hat{f}\|_q\ll_q N^{1-1/q}$, where the implied
constant depends only on $q$.
\item The function $f$ is said to be $(\beta,\kappa)$-regular if
\[ \sum_{(u,v)\in M}f(u)f(v)\geq\kappa N^2, \]
where
\[ M=\{(u,v):u\leq\beta N,v\geq (1-\beta)N,(v-u,P(\beta^{-1}))=1\}. \]
\end{enumerate}
\end{definition}

Note that when $f$ is the characteristic function of a subset
$A\subset [1,N]$, $(\beta,\kappa)$-regularity of $f$ is equivalent
to $(\beta,\kappa)$-regularity of $A$ (recall Definition
\ref{def:reg}).

The proof of Theorem \ref{thm:2/5} is similar as the arguments in
\cite{green_roth} and \cite{green_restriction}, but with some new
ingredients. In the treatment of the case $\nu_i=1$, we use Theorem
\ref{thm:pop} established in the previous section. In the reduction
from arbitrary $\nu_i$ to the case $\nu_i=1$, we work directly in
$\Z$ rather than in $\Z/N\Z$.

\subsection{Proof of Theorem \ref{thm:2/5} in the case $\nu_i=1$}

For clarity, we restate Theorem \ref{thm:2/5} in the case $\nu_i=1$
as the following lemma.

\begin{lemma}\label{lem:trans}
Let $0<\delta,\kappa<1$ be given. Let $N$ be a sufficiently large
positive integer. Let $N_3=N$ and $N_1=N_2=\lfloor N/2\rfloor$. For
$i=1,2,3$, let $a_i:[1,N_i]\rightarrow [0,1]$ be an arbitrary
function and let $\alpha_i$ be the average of $a_i$. Suppose that
they satisfy the following assumptions:
\begin{enumerate}
\item (mean condition) $\alpha_i\geq\delta$ and $\tfrac{1}{2}(\min(1,\alpha_1+\alpha_2)+\alpha_2)+\alpha_3\geq 1+\delta$.
\item (regularity condition for $a_1$) The function $a_1$ is
$(\delta/50,\kappa)$-regular.
\end{enumerate}
Then
\[ \sum_{\substack{n_1,n_2,n_3\\ n_1+n_2+n_3=N}}a_1(n_1)a_2(n_2)a_3(n_3)\geq c N^2, \]
where $c=c(\delta,\kappa)>0$ is a constant depending only on $\delta$ and $\kappa$.
\end{lemma}

\begin{proof}
Let $\xi>0$ be a small parameter to be chosen later. Let $A_i\subset
[1,N_i]$ be the essential support of $a_i$:
\[ A_i=\{1\leq n\leq N_i:a_i(n)\geq\xi\}, \]
Then
\[ |A_i|>(\alpha_i-\xi)N_i. \]
Write $\beta=\delta/50$. It follows from the regularity condition
for $a_1$ that
\[ |\{(u,v)\in A_1\times A_1:u\leq\beta N,v\geq (1-\beta)N,(v-u,P(\beta^{-1}))=1\}|\geq (\kappa-\xi^2\beta^2) N^2\geq\tfrac{1}{2}\kappa N^2 \]
if $\xi$ is chosen small enough. Hence $A_1$ is
$(\beta,\kappa/2)$-regular. By Theorem \ref{thm:pop}, there exists
$\gamma=\gamma(\delta,\kappa)>0$,
\[  |D_{\gamma N_1}(A_1,A_2)|\geq\min(N_1,|A_1|+|A_2|)+|A_2|-\tfrac{1}{2}\delta N_1\geq (\min(1,\alpha_1+\alpha_2)+\alpha_2-\delta)N_1. \]
Note that $D_{\gamma N_1}(A_1,A_2)$ and $A_3$ are both subsets of
$[1,N]$, and their densities in $[1,N]$ add up to at least
$1+\delta/4$ by the mean condition. Hence
\[ |D_{\gamma N_1}(A_1,A_2)\cap (N-A_3)|\geq\tfrac{1}{4}\delta N. \]
This shows that there are at least $\delta N/4$ ways to write $N$ as
the sum of an element in $D_{\gamma N_1}(A_1,A_2)$ and an element in
$A_3$. Each of these $\delta N/4$ representations gives rise to at
least $\gamma N_1$ ways to write $N$ as $a_1+a_2+a_3$ ($a_i\in
A_i$). This shows that
\[  \sum_{\substack{n_1,n_2,n_3\\ n_1+n_2+n_3=N}}a_1(n_1)a_2(n_2)a_3(n_3)\geq \xi^3\sum_{\substack{n_i\in A_i\\ n_1+n_2+n_3=N}}1\geq \tfrac{1}{4}\xi^3\delta\gamma N^2. \]
This completes the proof.

\end{proof}

\subsection{Decomposition of $a_i$ into uniform and anti-uniform parts}

For notational convenience, in this subsection we will fix some
$i\in\{1,2,3\}$ and simply write $a=a_i$, $\nu=\nu_i$, and $N=N_i$.
The main idea of reducing from general $\nu$ to the case $\nu=1$ is
to decompose the function $a$ into a structured part $a'$ and a
random part $a''$. The precise meanings of these properties are
summarized in Lemma \ref{lem:a'a''} below.

To construct this decomposition, let $0<\epsilon<1$ be a small
parameter to be chosen later (which depends only on $\delta$ and
$\kappa$). Let
\[ T=T_{\epsilon}=\{\theta\in\T:|\hat{a}(\theta)|\geq\epsilon N\}. \]
Since $a$ satisfies the $L^q$ extension estimate, the measure of
$T_{\epsilon}$ satisfies the bound
\begin{equation}\label{T}
\text{meas}(T_{\epsilon})\ll_{\epsilon}N^{-1}. \end{equation}
Define
\[ B=B_{\epsilon}=\{1\leq b\leq\epsilon N: \|b\theta\|<\epsilon\text{ for all }\theta\in T\}, \]
where $\|x\|$ denotes the distance from $x$ to its closest integer.
The definition of $B$ resembles the definition of Bohr sets in
finite abelian groups. In that setting, lower bounds for $|B|$ are
available in terms of its rank. The following lemma shows that a
similar lower bound holds in our situation as well.

\begin{lemma}\label{lem:large}
With the definitions of $T=T_{\epsilon}$ and $B=B_{\epsilon}$ as
above, we have $|B|\gg_{\epsilon}N$.
\end{lemma}

\begin{proof}
For each $\theta\in T_{\epsilon}$ and $\ell>0$, let $I(\theta,\ell)=[\theta-\ell/2,\theta+\ell/2]$ be the interval of length $\ell$ centered at $\theta$. By compactness, there exists $\theta_1,\cdots,\theta_m\in T$ so that
\[ T_{\epsilon}\subset I(\theta_1,\epsilon/24N)\cup\cdots\cup I(\theta_m,\epsilon/24N). \]
By the Vitali covering lemma, there exists a subcollection
$\{I(\theta_{j},\epsilon/24N):j\in J\}$ consisting of disjoint
intervals and satisfying
\begin{equation}\label{cover}
T_{\epsilon}\subset \bigcup_{j\in J}I(\theta_{j},\epsilon/8N).
\end{equation}

We claim that $|J|=O_{\epsilon}(1)$. In fact, for any $\theta\in
I(\theta_{j},\epsilon/8N)$ ($j\in J$),
\[
|\hat{a}(\theta)-\hat{a}(\theta_j)|\leq\sum_{n=1}^Na(n)|1-e(n(\theta-\theta_j))|\leq\sum_{n=1}^Na(n)\cdot\frac{\epsilon
n}{2N}\leq\frac{\epsilon}{2}\sum_{n=1}^N\nu(n)=\frac{1}{2}\epsilon
N.
\]
Hence
\[ \hat{a}(\theta)\geq\hat{a}(\theta_{j})-\tfrac{1}{2}\epsilon N\geq\tfrac{1}{2}\epsilon N. \]
It follows that
\[ \bigcup_{j\in J}I(\theta_{j},\epsilon/8N)\subset T_{\epsilon /2}. \]
Using \eqref{T} we get
\[ \frac{\epsilon |J|}{24N}=\sum_{j\in J}\text{meas}(I(\theta_j,\epsilon/24N))\leq\text{meas}(T_{\epsilon/2})\ll_{\epsilon}\frac{1}{N}. \]
This proves that $|J|=O_{\epsilon}(1)$.

Now let
\[ B'=\{1\leq b\leq\epsilon N:\|b\theta_{j}\|<\epsilon/2\text{ for all }j\in J\}. \]
We claim that $B'\subset B$. To see this, take any $b\in B'$ and
$\theta\in T_{\epsilon}$. By \eqref{cover}, $\theta\in
I(\theta_{j},\epsilon/8N)$ for some $j\in J$. Hence
\[ \|b\theta\|\leq \|b\theta_{j}\|+b|\theta_{j_k}-\theta|<\epsilon. \]
This shows that $B'\subset B$. A lower bound for $|B'|$ can be
obtained by a simple pigeonhole argument. Divide the $|J|$
dimensional cube $[0,1]^{|J|}$ into small cubes of side length
$\epsilon/2$. For each $1\leq b\leq\epsilon N$, consider the small
cube the vector $v_b=(\|b\theta_j\|)_{j\in J}$ belongs to. By the
pigeonhole principle, there exists a small cube containing at least
$(2/\epsilon)^{|J|}\epsilon N$ vectors $v_b$. For $b_1,b_2$ with
$v_{b_1},v_{b_2}$ in the same small cube, the difference $|b_1-b_2|$
is an element of $B'$. Hence $|B|\geq |B'|\gg_{\epsilon}N$.
\end{proof}

The remaining arguments go along the same line as those of Green
\cite{green_roth,green_restriction}. Define
\[ a'(n)=\E_{b_1,b_2\in B}a(n+b_1-b_2)=\frac{1}{|B|^2}\sum_{b_1,b_2\in B}a(n+b_1-b_2),\quad a''(n)=a(n)-a'(n). \]

\begin{lemma}\label{lem:a'a''}
Suppose that $\eta$ is chosen small enough depending on $\epsilon$.
The functions $a'$ and $a''$ defined above have the following
properties:
\begin{enumerate}
\item ($a'$ is set-like) $0\leq a'(n)\leq 1+O_{\epsilon}(\eta)$ for any $n$. Moreover, $\E_{1\leq n\leq
N}a'(n)=\alpha+O(\epsilon)$.
\item ($a''$ is uniform)  $\hat{a}''(\theta)=O(\epsilon N)$ for all $\theta$.
\item ($a_1'$ is regular) $a_1'$ is
$(\delta/50,\kappa-O(\epsilon))$-regular.
\item $\|\hat{a}'\|_q\leq\|\hat{a}\|_q$ and
$\|\hat{a}''\|_q\leq\|\hat{a}\|_q$.
\end{enumerate}
\end{lemma}

\begin{proof}
To prove (1), note that
\begin{align*}
a'(n) &\leq \E_{b_1,b_2\in B}\nu(n+b_1-b_2)\leq \E_{b_1,b_2\in B}\E_{0\leq r<N}\hat{\nu}(r/N)e_N(r(n+b_1-b_2)) \\
&=\E_{0\leq r<N}\hat{\nu}(r/N)e_N(rn)|\E_{b\in B}e_N(rb)|^2.
\end{align*}
The term $r=0$ gives $\hat{\nu}(0)=N(1+O(\eta))$. For $r\neq 0$, the
summand is bounded in absolute value by $\eta N|\E_{b\in
B}e_N(rb)|^2$. Hence
\[ a'(n)\leq 1+O(\eta)+\eta N\E_{0\leq r<N}|\E_{b\in B}e_N(rb)|^2=1+O(\eta)+\eta N|B|^{-1} \]
by Parseval's identity. By Lemma \ref{lem:large},
\[ a'(n)\leq 1+O_{\epsilon}(\eta). \]
If $\eta$ is chosen sufficiently small, $a'(n)\leq 2$ for all $n$.
The fact that $\E_{1\leq n\leq N}a'(n)=\alpha+O(\epsilon)$ follows
since $\E_{n\in\Z}a'(n)=\alpha$ and the support of $a'$ is contained
in $[-\epsilon N,(1+\epsilon)N]$.

To prove (2), note that the Fourier transform of $a''$ can be written as
\[ \hat{a}''(\theta)=\hat{a}(\theta)\left(1-|\E_{b\in B}e(b\theta)|^2\right). \]
For $\theta\notin T$, $|\hat{a}''(\theta)|\leq
|\hat{a}(\theta)|\leq\epsilon N$. For $\theta\in T$, we have
\[ 1-|\E_{b\in B}e(b\theta)|^2\leq 2(1-|\E_{b\in B}e(b\theta)|)\leq 2\E_{b\in B}|1-e(b\theta)|\ll\epsilon \]
by the definition of $B$. Hence $|\hat{a}''(\theta)|\ll\epsilon N$
as well.

To prove (3), write $\beta=\delta/50$. Define
\[ M=\{(u,v):1\leq u\leq \beta N,(1-\beta)N\leq v\leq
N,(v-u,P(\beta^{-1}))=1\}, \] and
\[ M'=\{(u,v):-\epsilon N\leq u\leq (\beta+\epsilon) N,(1-\beta-\epsilon)N\leq v\leq (1+\epsilon)N,(v-u,P(\beta^{-1}))=1\}. \]
Note that
\begin{align*}
\sum_{(u,v)\in M'}a'(u)a'(v) &=\E_{b_1,b_2,b_3,b_4\in B}\sum_{(u,v)\in M'}a(u+b_1-b_2)a(v+b_3-b_4) \\
&\geq\E_{b_1,b_2,b_3,b_4\in B}\sum_{(u,v)\in M}a(u)a(v)\geq\kappa
N^2.
\end{align*}
Hence,
\[ \sum_{(u,v)\in M}a'(u)a'(v)\geq\sum_{(u,v)\in M'}a'(u)a'(v)-2|M'\setminus M|\geq (\kappa-O(\epsilon))N^2. \]

To prove (4), note that for any $\theta$,
\[ \hat{a}'(\theta)=\hat{a}(\theta)|\E_{b\in B}e(b\theta)|^2,\quad \hat{a}''(\theta)=\hat{a}(\theta)(1-|\E_{b\in B}e(b\theta)|^2), \]
and thus $|\hat{a}'(\theta)|\leq |\hat{a}(\theta)|$ and
$|\hat{a}''(\theta)|\leq |\hat{a}(\theta)|$.

\end{proof}

\subsection{Reduction to the case $\nu_i=1$}\label{sec:reduce}

For each $i\in\{1,2,3\}$, we obtained a decomposition
$a_i=a_i'+a_i''$ satisfying the conditions summarized in Lemma
\ref{lem:a'a''}. In this section, we will show that the
contributions from $a_i''$ are negligible, and thus we may
essentially replace $a_i$ by $a_i'$. Now that the functions $a_i'$
are essentially bounded above by $1$, we are back in the case
$\nu_i=1$ treated in Lemma \ref{lem:trans}.

\begin{lemma}\label{lem:diff}
With the functions $a_i,a_i'$ defined as above, we have
\[ \left|\sum_{n,m}a_1(n)a_2(m)a_3(N-n-m)-\sum_{n,m}a_1'(n)a_2'(m)a_3'(N-n-m)\right|\ll \epsilon^{3-q} N^2. \]
\end{lemma}

\begin{proof}
The difference on the left can be expressed as a sum of several terms, each of the form
\[ \sum_{\substack{n_1,n_2,n_3\\ n_1+n_2+n_3=N}}f_1(n_1)f_2(n_2)f_3(n_3)=\int_0^1\hat{f_1}(\theta)\hat{f_2}(\theta)\hat{f_3}(\theta)e(-N\theta)d\theta, \]
where $f_i\in\{a_i,a_i',a_i''\}$, and $f_i=a_i''$ for at least one $i$. Without loss of generality, assume that $f_3=a_3''$. By H\"{o}lder's inequality, this is bounded above by
\[ \|\hat{f}_3\|_{\infty}^{3-q}\|\hat{f}_3\|_q^{q-2}\|\hat{f}_1\|_q\|\hat{f}_2\|_q. \]
By Lemma \ref{lem:a'a''}, $\|\hat{f}_3\|_{\infty}\ll\epsilon N$. By
the $L^q$ extension estimate together with Lemma \ref{lem:a'a''},
all of $\|\hat{f}_3\|_q$, $\|\hat{f}_1\|_q$, and $\|\hat{f}_2\|_q$
are bounded above by $O_q(N^{1-1/q})$. Combining these we get the
desired bound.
\end{proof}

We now finish the proof of Theorem \ref{thm:2/5}. By Lemma
\ref{lem:a'a''}, the functions $a_i'$ are all bounded above
uniformly by $1+O_{\epsilon}(\eta)$ with averages
$\alpha+O(\epsilon)$, and $a_1'$ is $(\delta/50,\kappa/2)$-regular.
If $\epsilon$ and $\eta$ are chosen small enough, Lemma
\ref{lem:trans} then implies that
\[ \sum_{n,m}a_1'(n)a_2'(m)a_3'(N-n-m)\geq cN^2 \]
for some $c=c(\delta,\kappa)>0$. Combining this with Lemma \ref{lem:diff}, we deduce by choosing $\epsilon$ small enough that
\[ \sum_{n,m}a_1(n)a_2(m)a_3(N-n-m)\geq \tfrac{1}{2}cN^2. \]
This completes the proof of Theorem \ref{thm:2/5}.

%%%%%%%%%%%%%%%%%%%%%%%%%%%%%%%%%%%%%%%%%%%%%%%%%%%%%%%%%%%%%%%%%%%%%%%%%%%
%%%%%%%%%%%%%%%%%%% CONSTRUCTION OF SELBERG MAJORANT %%%%%%%%%%%%%%%%%%%%%%
%%%%%%%%%%%%%%%%%%%%%%%%%%%%%%%%%%%%%%%%%%%%%%%%%%%%%%%%%%%%%%%%%%%%%%%%%%%

\section{Construction of Selberg's majorant}\label{sec:selberg}

In this section and the next, we begin the task of constructing
Selberg's majorant $\nu$ and showing that it is pseudorandom. The
construction of $\nu$ can be found in any book on sieve theory. Our
notations will follow those in \cite{opera}. After recalling some
classical properties of Selberg's weights, we prove the main result
in this section, Lemma \ref{lem:Tq}, which will be used to show that
$\nu$ is pseudorandom.

Let $W$ be a squarefree positive integer and let $b\pmod W$ be a
reduced residue modulo $W$. Apply the arguments in Chapter 7 of
\cite{opera} to construct an upper bound sieve for the primes $Wn+b$
($1\leq n\leq N$). See, in particular, Theorem 7.1 in \cite{opera}.

Let $z\geq 2$ be a parameter and let $D=z^2$ be the sieving level.
Let $P$ be the product of all primes $p<z$ with $(p,W)=1$. The
weights $\rho_d$ defined for $d\mid P$ satisfy the following
properties. They are supported on integers smaller than $z$:
\begin{equation}\label{support}
\rho_d=0\ \ \text{if }d\geq z.
\end{equation}
Their absolute values are bounded:
\begin{equation}\label{rhod}
|\rho_d|\leq 1.
\end{equation}
Moreover,
\begin{equation}\label{rho1}
\rho_1=1. \end{equation}
The linear change of variables
\begin{equation}\label{y} y_d=\mu(d)\phi(d)\sum_{m\equiv 0\pmod
d}\frac{\rho_m}{m}
\end{equation}
satisfy $y_d=J^{-1}$ for $d<z$ and $y_d=0$ for $d\geq
z$, where
\[ J=\sum_{\substack{d\mid P\\ d<z}}\frac{1}{\phi(d)}=\sum_{\substack{d<z\\ (d,W)=1}}\frac{1}{\phi(d)}. \]

Using these weights we define a function
$\nu=\nu(N,z,W,b):[1,N]\rightarrow\R$ by
\[ \nu(n)=\frac{\phi(W)}{W}\log z\left(\sum_{d\mid
(Wn+b,P)}\rho_d\right)^2. \] Clearly, \eqref{support} and
\eqref{rho1} imply that
\begin{equation}\label{major}
\nu(n)\geq\frac{\phi(W)}{W}\log z\ \ \text{if }Wn+b\text{ is prime
and }Wn+b\geq z.
\end{equation}
Hence $\nu$ serves as a majorant for the primes of the form $Wn+b$.

The following estimate will be used multiple times:

\begin{lemma}\label{lem:phi}
For any $z\geq 2$ and positive integer $m$ dividing $P(z)$,
\[ \sum_{\substack{d<z\\
(d,m)=1}}\frac{1}{\phi(d)}\ll\frac{\phi(m)}{m}\log z \] and
\[ \sum_{\substack{d<z\\
(d,m)=1}}\frac{1}{\phi(d)}=\frac{\phi(m)}{m}(\log z+O_m(1)).
\]
\end{lemma}

\begin{proof}
The upper bound is clear:
\[ \sum_{\substack{d<z\\
(d,m)=1}}\frac{1}{\phi(d)}\leq\prod_{\substack{p<z\\ p\nmid
m}}\left(1+\frac{1}{\phi(p)}\right)=\frac{\phi(m)}{m}\prod_{p<z}\left(1+\frac{1}{p-1}\right)\ll\frac{\phi(m)}{m}\log
z.
\]
The asymptotic can be obtained by standard methods in analytic
number theory. See, for example, Theorem A.8 in \cite{opera}.
\end{proof}

In particular, Lemma \ref{lem:phi} implies
\begin{equation}\label{J}
J=\frac{\phi(W)}{W}(\log z+O_W(1)).
\end{equation}

\begin{lemma}\label{lem:Jqr}
For any positive integers $q$ and $r$ dividing $P$, the sum
\[ J(q,r)=\sum_{\substack{d\mid P\\ (d,q)=1}}\frac{\rho_{rd}}{d} \]
satisfies \[ |J(q,r)|\leq J^{-1}\frac{[q,r]}{\phi([q,r])}. \]
Moreover, $J(q,q)=y_qq\mu(q)/\phi(q)$.
\end{lemma}

\begin{proof}
We write
\[ J(q,r)=\sum_{d\mid
P}\frac{\rho_{rd}}{d}\sum_{e\mid (d,q)}\mu(e)=\sum_{e\mid
q}\mu(e)\sum_{\substack{d\mid P\\ e\mid d}}\frac{\rho_{rd}}{d}. \]
Note that $\rho_{rd}=0$ if $rd$ is not squarefree. Hence we can
restrict the sum to those $e$ with $(e,r)=1$:
\[ J(q,r)=\sum_{e\mid q/(q,r)}\mu(e)\sum_{\substack{d\mid P\\ e\mid
d}}\frac{\rho_{rd}}{d}=r\sum_{e\mid
q/(q,r)}\mu(e)y_{re}\mu(re)\phi(re)^{-1}=\frac{r\mu(r)}{\phi(r)}\sum_{e\mid
q/(q,r)}\frac{y_{re}}{\phi(e)}.
\]
If $q=r$, then $q/(q,r)=1$, and thus
\[ J(q,q)=\frac{q\mu(q)}{\phi(q)}y_q. \]
In general, since $y_{re}$ is bounded by $J^{-1}$, it follows that
\[ |J(q,r)|\leq J^{-1}\frac{r}{\phi(r)}\sum_{e\mid
q/(q,r)}\frac{1}{\phi(e)}=J^{-1}\frac{r}{\phi(r)}\frac{q/(q,r)}{\phi(q/(q,r))}=J^{-1}\frac{[q,r]}{\phi([q,r])}.
\]
\end{proof}

\begin{lemma}\label{lem:Tq}
For any positive integer $q$ dividing $P$, the sum
\[ T(q)=\sum_{\substack{d_1,d_2\mid P\\ q\mid
[d_1,d_2]}}\frac{\rho_{d_1}\rho_{d_2}}{[d_1,d_2]} \] satisfies
\[ |T(q)|\ll_{\epsilon} J^{-1}q^{-1+\epsilon}. \]
Moreover, $T(1)=J^{-1}$.
\end{lemma}

\begin{proof}
Write $e_1=(d_1,q)$, $d_1=e_1f_1$, $e_2=(d_2,q)$, and $d_2=e_2f_2$.
Then
\[ T(q)=\sum_{\substack{e_1,e_2\mid q\\
[e_1,e_2]=q}}\sum_{\substack{f_1,f_2\mid P\\
(f_1,q)=(f_2,q)=1}}\frac{\rho_{e_1f_1}\rho_{e_2f_2}}{q[f_1,f_2]} \]
For fixed $e_1,e_2$, use the identities $(f_1,f_2)[f_1,f_2]=f_1f_2$
and $(f_1,f_2)=\sum_{g\mid (f_1,f_2)}\phi(g)$ to rewrite the inner
sum as
\[ \frac{1}{q}\sum_{\substack{f_1,f_2\mid P\\
(f_1,q)=(f_2,q)=1}}\frac{\rho_{e_1f_1}\rho_{e_2f_2}}{f_1f_2}\sum_{g\mid
(f_1,f_2)}\phi(g)=\frac{1}{q}\sum_{\substack{g\mid P\\
(g,q)=1}}\phi(g)\left(\sum_{\substack{f_1\mid P\\ (f_1,q)=1\\ g\mid
f_1}}\frac{\rho_{e_1f_1}}{f_1}\right)\left(\sum_{\substack{f_2\mid
P\\ (f_2,q)=1\\ g\mid f_2}}\frac{\rho_{e_2f_2}}{f_2}\right)
\]
The two sums in the parentheses above are $g^{-1}J(qg,e_1g)$ and
$g^{-1}J(qg,e_2g)$. When $q=1$, apply Lemma \ref{lem:Jqr} to get
\[ T(1)=\sum_{g\mid P}\phi(g)(g^{-1}J(g,g))^2=\sum_{g\mid
P}\frac{y_g^2}{\phi(g)}=J^{-2}\sum_{\substack{g\mid P\\
g<z}}\frac{1}{\phi(g)}=J^{-1}. \]

In general, Lemma \ref{lem:Jqr} gives the bounds
\[ |g^{-1}J(qg,e_1g)|\leq J^{-1}\frac{q}{\phi(qg)},\ \
|g^{-1}J(qg,e_2g)|\leq J^{-1}\frac{q}{\phi(qg)}. \] Observe that
there are $3^{\omega(q)}$ pairs $(e_1,e_2)$ with $[e_1,e_2]=q$. Note
also that we can clearly restrict the sum to $g<z$. Hence by Lemma
\ref{lem:phi},
\[ |T(q)|\leq 3^{\omega(q)}J^{-2}\frac{q}{\phi(q)^2}\sum_{\substack{g<z\\ (g,qW)=1}}\frac{1}{\phi(g)}\ll 3^{\omega(q)}J^{-2}\frac{q}{\phi(q)^2}\frac{\phi(qW)}{qW}\log z\ll J^{-1}\frac{3^{\omega(q)}}{\phi(q)}. \]
The desired bound for $|T(q)|$ follows because
$3^{\omega(q)}\ll_{\epsilon}q^{\epsilon}$ and $\phi(q)\gg_{\epsilon}
q^{1-\epsilon}$.
\end{proof}

%%%%%%%%%%%%%%%%%%%%%%%%%%%%%%%%%%%%%%%%%%%%%%%%%%%%%%%%%%%%%%%%%%%%%
%%%%%%%%%%%%%%%% PSEUDORANDOMNESS OF MAJORANT %%%%%%%%%%%%%%%%%%%%%%%
%%%%%%%%%%%%%%%%%%%%%%%%%%%%%%%%%%%%%%%%%%%%%%%%%%%%%%%%%%%%%%%%%%%%%

\section{Pseudorandomness of Selberg's majorant}\label{sec:random}

Let $W=\prod_{p\leq w}p$ be the product of primes up to some large
constant $w$. Let $b\pmod W$ be a reduced residue class. Fix a small
positive constant $\delta>0$. Let $N$ be a positive integer
sufficiently large (depending on $w$ and $\delta$) and take
$z=N^{1/2-\delta}$. In the previous section, we constructed a
majorant $\nu=\nu(N,z,W,b):[1,N]\rightarrow\R$ by
\[ \nu(n)=\frac{\phi(W)}{W}\log z\left(\sum_{d\mid
(Wn+b,P)}\rho_d\right)^2 \] with some weights $\rho_d$. In this
section, we show that $\nu$ is pseudorandom.

\begin{theorem}\label{thm:random}
With the notations as above, for any $r\in\Z/N\Z$,
\[ \hat{\nu}(r)=(\delta_{r,0}+O_{\epsilon}(w^{-1+\epsilon}))N, \]
where $\delta_{r,0}$ is the Kronecker delta. In other words, $\nu$
is $O_{\epsilon}(w^{-1+\epsilon})$-pseudorandom.
\end{theorem}

The proof proceeds by considering two cases depending on whether or
not $r/N$ is close to a rational with small denominator. Let
$R=\lfloor N^{1-\delta/2}\rfloor$ and $Q=\lfloor
N^{\delta/4}\rfloor$ be parameters. For $q\leq Q$ and $(a,q)=1$, let
\[
\mathfrak{M}(q,a)=\left\{r\in\Z/N\Z:\left|\frac{r}{N}-\frac{a}{q}\right|\leq\frac{1}{qR}\right\}.
\]
Let
\[ \mathfrak{M}=\bigcup_{q=1}^Q\bigcup_{\substack{a=1\\
(a,q)=1}}^q\mathfrak{M}(q,a),\ \
\mathfrak{m}=\Z/N\Z\setminus\mathfrak{M}. \]

\subsection{Major arc analysis}

In this subsection, we prove Theorem \ref{thm:random} for those
$r\in\mathfrak{M}$. Suppose that $r\in\mathfrak{M}(q,a)$ for some
$q\leq Q$ and $(a,q)=1$. Then $r/N$ is very close to $a/q$. We first
prove a result when they are equal. Recall the quantity $T(q)$
defined in Lemma \ref{lem:Tq}.

\begin{proposition}\label{prop:major}
With the notations as above, for $1\leq x\leq N$,
\[ f(x,a/q)=\sum_{n\leq x}\nu(n)e_q(an)=\frac{\phi(W)}{W}\log z(\varepsilon xT(q)+E(x,q)), \]
where $\varepsilon=\varepsilon(a/q,W,b)$ does not depend on $x$, and
$E(x,q)=O(qN^{1-\delta})$. Moreover, $\varepsilon=1$ if $q=1$,
$\varepsilon=0$ if $(q,W)>1$, and $|\varepsilon|=1$ if $(q,W)=1$.
\end{proposition}

\begin{proof}
By the definition of $\nu(n)$, we can write
\[
f(x,a/q)=\frac{\phi(W)}{W}\log z\sum_{d_1,d_2\mid
P}\rho_{d_1}\rho_{d_2}\sum_{\substack{n\leq x\\ [d_1,d_2]\mid
Wn+b}}e_q(an).
\]
Split the sum into two parts:
\[ f(x,a/q)=\frac{\phi(W)}{W}\log z(S_1+S_2), \]
where
\[ S_1=\sum_{\substack{d_1,d_2\mid P\\ q\mid [d_1,d_2]}}\rho_{d_1}\rho_{d_2}\sum_{\substack{n\leq x\\ [d_1,d_2]\mid Wn+b}}e_q(an), \]
\[ S_2=\sum_{\substack{d_1,d_2\mid P\\ q\nmid [d_1,d_2]}}\rho_{d_1}\rho_{d_2}\sum_{\substack{n\leq x\\ [d_1,d_2]\mid Wn+b}}e_q(an). \]

First consider $S_1$. For $q\mid [d_1,d_2]$, the inner sum is zero
if $(q,W)>1$. Take $\epsilon=0$ in the case $(q,W)>1$. If $(q,W)=1$,
then the summand in the inner sum is a constant $\epsilon$ with
$|\epsilon|=1$. Moreover, $\epsilon=1$ when $q=1$. In either case,
\[ S_1=\epsilon\sum_{\substack{d_1,d_2\mid P\\ q\mid [d_1,d_2]}}\rho_{d_1}\rho_{d_2}\left(\frac{x}{[d_1,d_2]}+O(1)\right)=\epsilon xT_q+O\left(\left(\sum_{d<z}|\rho_d|\right)^2\right)=\epsilon xT_q+O(N^{1-\delta}) \]
since $|\rho_d|\leq 1$ by \eqref{rhod}.

Now consider $S_2$.  For $d_1,d_2\leq z$ with $(d_1d_2,W)=1$ and $q\nmid [d_1,d_2]$, the inner sum over $n$ is bounded by $q$. Hence
\[ S_2\leq q\left(\sum_{d\leq z}|\rho_d|\right)^2\leq qN^{1-\delta}.
\]
The proof is completed by combining the estimates for $S_1$ and
$S_2$.
\end{proof}

We now use partial summation to complete the major arc estimate. Let $r\in\mathcal{M}(q,a)$ for some $q\leq Q$ and $(a,q)=1$.
Then $r/N=a/q+\beta$ for some $|\beta|\leq 1/qR$. Note that
\[ \hat{\nu}(r)=\sum_{n=1}^N\nu(n)e_q(an)e(\beta n)=\int_1^Ne(\beta x)d\left(\sum_{n\leq x}\nu(n)e_q(an)\right). \]
It follows from Proposition \ref{prop:major} that
\[ \hat{\nu}(r)=\frac{\phi(W)}{W}\log z\left(\epsilon T(q)\int_1^Ne(\beta x)dx+\int_1^Ne(\beta x)dE(x,q)\right). \]

Consider the second integral above. By partial summation, it is
bounded by
\[ E(N,q)+\int_1^NE(x,q)(2\pi i\beta)e(\beta x)dx\ll qN^{1-\delta}+|\beta|qN^{2-\delta}\leq QN^{1-\delta}+\frac{N^{2-\delta}}{R}. \]
This is $O(N^{1-\delta/2})$ by the choices of $Q$ and $R$. Hence
\[ \hat{\nu}(r)=\frac{\phi(W)}{W}\log z\left(\epsilon T(q)\int_1^Ne(\beta x)dx+O(N^{1-\delta/2})\right). \]
If $q>w$, then Theorem \ref{thm:random} follows from Lemma
\ref{lem:Tq} and \eqref{J}. If $1<q\leq w$, then $(q,W)>1$ and thus
$\epsilon=0$. If $q=1$ and $\beta>0$, then $\beta$ is an integer
multiple of $1/N$, and thus the integral above is zero. Finally, if
$q=1$ and $\beta=0$, then $\epsilon=1$. Lemma \ref{lem:Tq} and
\eqref{J} give
\[ \hat{\nu}(0)=\frac{\phi(W)}{W}\log z(J^{-1}+O(N^{-\delta/2}))N=(1+O_W((\log z)^{-1}))N. \]
This proves Theorem \ref{thm:random} for sufficiently large $z$.

\subsection{Minor arc analysis}
Now consider the case when $r\in\mathfrak{m}$. This means that
\[ \left|\frac{r}{N}-\frac{a}{q}\right|\leq\frac{1}{q^2}, \]
for some $Q\leq q\leq R$ and $(a,q)=1$.

By the definition of $\nu(n)$, we can write
\[ \hat{\nu}(r)=\frac{\phi(W)}{W}\log z\sum_{d_1,d_2\mid P}\rho_{d_1}\rho_{d_2}\sum_{\substack{1\leq n\leq N\\ [d_1,d_2]\mid Wn+b}}e_N(rn). \]
Using \eqref{rhod} we obtain
\[ |\hat{\nu}(r)|\leq\frac{\phi(W)}{W}\log z\sum_{\substack{d<
z^2\\ (d,W)=1}}\left(\sum_{\substack{d_1,d_2\\
[d_1,d_2]=d}}1\right)\left(\sum_{\substack{1\leq n\leq N\\
d\mid Wn+b}}e_N(rn)\right). \] For any fixed squarefree $d<z^2$,
there are at most $3^{\omega(d)}\ll d^{\delta/8}$ pairs $(d_1,d_2)$
with $[d_1,d_2]=d$. Hence
\[ |\hat{\nu}(r)|\ll\frac{\phi(W)}{W}(\log
z)N^{\delta/8}\sum_{d<z^2}\left|\sum_{\substack{1\leq n\leq N\\
d\mid Wn+b}}e_N(rn)\right|.
\]

The following lemma estimates this double sum.

\begin{lemma}
Suppose that
\[ \left|\alpha-\frac{a}{q}\right|\leq\frac{1}{q^2} \]
with $(a,q)=1$. For any $1\leq m\leq M$, let $c_m\pmod m$ be an
arbitrary residue class. Then
\[ \sum_{1\leq m\leq M}\left|\sum_{\substack{1\leq n\leq x\\ n\equiv c_m\pmod m}}e(\alpha n)\right|\ll (M+xq^{-1}+q)\log (2qx). \]
\end{lemma}

\begin{proof}
See Lemma 13.7 in \cite{iwaniec}.
\end{proof}

It follows that
\[ |\hat{\nu}(r)|\ll\frac{\phi(W)}{W}(\log
z)N^{\delta/8}(z^2+NQ^{-1}+R)\log N\ll N^{1-\delta/4}, \] completing
the proof of Theorem \ref{thm:random} in the minor arc case.

%%%%%%%%%%%%%%%%%%%%%%%%%%%%%%%%%%%%%%%%%%%%%%%%%%%%%%%%%%%%%%%%%%%
%%%%%%%%%%%%%%%%%%%%%% PROOF OF VINOGRADOV %%%%%%%%%%%%%%%%%%%%%%%%
%%%%%%%%%%%%%%%%%%%%%%%%%%%%%%%%%%%%%%%%%%%%%%%%%%%%%%%%%%%%%%%%%%%

\section{Proof of Vinogradov's Theorem}\label{sec:proof}

In this section, we use Theorem \ref{thm:2/5} with Selberg's
majorant considered in Sections \ref{sec:selberg} and
\ref{sec:random} to give a proof of Vinogradov's three primes
theorem without using the theory of $L$-functions. In particular, we
will not need Siegel-Walfisz theorem, although we still use the
prime number theorem in arithmetic progressions with constant
modulus, which can be proved elementarily.

Let $M$ be a sufficiently large odd positive integer. We will prove
that $M$ can be written as sum of three primes. Take $\delta=0.01$
in the statement of Theorem \ref{thm:2/5}. Let $W=P(w)$ be a
parameter to be chosen later. Choose $0<b_1,b_2,b_3<W$ with
$(b_i,W)=1$ such that $b_1+b_2+b_3\equiv M\pmod W$ (this can always
be done by the Chinese Remainder Theorem).  Let
$N=(M-b_1-b_2-b_3)/W$. Let $N_3=N$ and $N_1=N_2=\lfloor N/2\rfloor$.
For $i=1,2,3$, define a function $a_i:[1,N_i]\rightarrow\R$ by
\[ a_i(n)=\begin{cases} \frac{\phi(W)}{W}\log z_i & Wn+b_i\text{ is prime and }Wn+b_i\geq z_i \\ 0 & \text{otherwise,} \end{cases} \]
where $z_i=N_i^{0.49}$. Construct $\nu_i=\nu(N_i,z_i,W,b_i)$ as in
Section \ref{sec:selberg}.

The majorization condition is satisfied by the observation
\eqref{major}. The mean condition is satisfied because the average
of $a_i$ is at least $0.48$ for sufficiently large $N$ by the prime
number theorem. The pseudorandomness condition is satisfied by
Theorem \ref{thm:random}, if $w$ is chosen large enough.

Now consider the regularity condition for $a_1$. Write
$\beta=\delta/50$, $y=\beta^{-1}$, $Y=P(y)$, and let
\[ M=\{(u,v):u\leq\beta N_1,v\geq (1-\beta)N_1,(v-u,Y)=1\}.
\]
Also write
\[ U=\{1\leq u\leq\beta N_1:Wu+b_1\text{ is prime}\},\ \
V=\{(1-\beta)N_1\leq v\leq N_1:Wv+b_1\text{ is prime}\}. \]
\begin{align*}
\sum_{(u,v)\in M}a_1(u)a_1(v) &=\left(\frac{\phi(W)}{W}\log z_1\right)^2\sum_{\substack{u\in U,v\in V\\ (v-u,Y)=1}}1 \\
&\geq \left(\frac{\phi(W)}{W}\log z_1\right)^2\sum_{\substack{s_1,s_2\pmod Y\\ (s_2-s_1,Y)=1}}|U\cap (Y\Z+s_1)|\cdot |V\cap (Y\Z+s_2)| \\
&\geq \left(\frac{\phi(W)}{W}\log
z_1\right)^2Y\phi(Y)\left(\frac{\beta N_1}{2\log
N_1}\cdot\frac{W}{\phi(W)}\cdot\frac{1}{Y}\right)^2\geq\kappa N^2
\end{align*}
for some $\kappa$ depending only on $\delta$.

Finally, the $L^q$ extension estimate for $a_i$ follows from the
result of Green \cite{green_restriction}.

\begin{lemma}\label{lem:lq}
For any $q>2$,
\[ \left(\int_0^1|\hat{a}_i(\theta)|^qd\theta\right)^{1/q}\ll_q N^{1-1/q}. \]
\end{lemma}

\begin{proof}
Consider the linear function $F(n)=Wn+b_i$ and the exponential sum
\[ h(\theta)=\sum_{\substack{n\leq N_i\\ F(n)\geq z_i\\ F(n)\text{
prime}}}e(n\theta). \] The argument leading to Theorem 1.1 in
\cite{green_restriction} gives
\[ \|h\|_q\ll_q\mathfrak{G}_FN_i^{1-1/q}(\log N_i)^{-1}, \]
where the singular series $\mathfrak{G}_F$ is defined by
\[ \mathfrak{G}_F=\prod_{p\text{
prime}}\gamma(p)\left(1-\frac{1}{p}\right)^{-1} \] and
\[ \gamma(p)=p^{-1}|\{n\in\Z/p\Z:(p,F(n))=1\}|. \]
(See (1.2) and (1.7) in \cite{green_restriction}). In the current
case, $\gamma(p)=1$ for $p\leq w$ and $\gamma(p)=1-1/p$ for $p>w$.
Hence
\[ \mathfrak{G}_F=\prod_{p\leq w}\frac{p}{p-1}=\frac{W}{\phi(W)}. \]
Finally, note that
\[ \hat{a}_i(\theta)=\left(\frac{\phi(W)}{W}\log
z_i\right)h(\theta). \] It follows that
\[ \|\hat{a}_i\|_q\leq\left(\frac{\phi(W)}{W}\log z_i\right)\|h\|_q\ll_q
N_i^{1-1/q}. \]

\end{proof}

\begin{remark}\label{rem:Lq}
Lemma \ref{lem:lq} was also proved in \cite{green_roth}, using the
Brun sieve and the Siegel-Walfisz theorem. Bourgain \cite{bourgain}
showed how to obtain bounds for $\|\hat{f}\|_q$, where $f$ is a
function supported on the primes. The proof in
\cite{green_restriction} differs from these previous arguments, and
solely depends on properties of an enveloping sieve; in particular
the theory of $L$-functions is not used.
\end{remark}

Now that all hypotheses in the statement of Theorem \ref{thm:2/5}
are verified, we conclude that there exists $n_i\in [1,N_i]$ with
$a_i(n_i)>0$ such that $N=n_1+n_2+n_3$. In particular, $Wn_i+b_i$ is
prime and
\[ M=WN+b_1+b_2+b_3=(Wn_1+b_1)+(Wn_2+b_2)+(Wn_3+b_3), \]
proving that $M$ is the sum of three primes.

We make a final remark concerning the explicit bound for $M$ that
can be produced from our method. As mentioned in the introduction,
directly following our arguments gives $M\geq\exp(\exp(\exp(C)))$
for some reasonable constant $C$. This can be seen as follows. For
our choice of $\delta$, the transference principle Theorem
\ref{thm:2/5} requires the parameter $\eta$ to be exponential in
$1/\delta$. Thus by the pseudorandomness estimate Theorem
\ref{thm:random}, the parameter $w$ should be taken to be
exponential in $1/\delta$. Hence $W$, being the product of primes up
to $w$, becomes double exponential in $1/\delta$. Finally, in the
arguments in this section we used lower bounds on the number of
primes up to $M$ in congruence classes modulo $W$. Such lower bounds
are only available when $M$ is exponential in $W$, and thus triple
exponential in $1/\delta$.

\bibliographystyle{plain}
\bibliography{vinogradov}{}

\end{document}